\documentclass[preprint]{imsart}

\RequirePackage[OT1]{fontenc}
\RequirePackage{amsthm,amsmath}
\RequirePackage[numbers]{natbib}
\RequirePackage[colorlinks,citecolor=blue,urlcolor=blue]{hyperref}

\startlocaldefs
\numberwithin{equation}{section}
\theoremstyle{plain}
\newtheorem{thm}{Theorem}[section]
\newtheorem{lem}{Lemma}[section]
\newtheorem{defn}{Definition}[section]
\newtheorem{rem}{Remark}[section]
\newtheorem{cor}{Corollary}[section]
\newtheorem{prob}{Problem}[section]

\endlocaldefs

\begin{document}

\begin{frontmatter}
\title{Under Collatz conjecture the Collatz mapping has no an asymptotic   mixing  property $\pmod{3}$}%\thanksref{T1}}
\runtitle{Under Collatz conjecture the Collatz mapping has no $\cdots$}
%\thankstext{T1}{Footnote to the title with the `thankstext' command.}

\begin{aug}
\author{\fnms{Gogi} \snm{Pantsulaia}\thanksref{t1}\ead[label=e1]{g.pantsulaia@gtu.ge}}
%\and
%$\author{\fnms{Second} \snm{Author}\thanksref{t3}\ead[label=e2]{second@somewhere.com}}
\affiliation{Georgian Technical University\\
 I.Vekua Institute of Applied Mathematics }%\thanksmark{m1}}

\thankstext{t1}{The author partially is
supported by Shota Rustaveli National Science Foundation's Grant
no 31/25.}
\runauthor{G.Pantsulaia}

\address{Kostava Street. 77 , Tbilisi
DC  0175, Georgian Republic \\
\printead{e1}}
%\phantom{E-mail:\ }\printead*{}}
\end{aug}

\begin{abstract}By using  properties of Markov  homogeneous chains, we compute Banach measure of the set of all natural numbers which under Collatz motion after  $n$ step will visit the set of even numbers in $\mathrm{N}$. As consequence, we get that the relative frequency
of even numbers in the sequence of $n$-th coordinates of all Collatz sequences is equal to the number
$
\frac{2}{3}+\frac{(-1)^{n+1}}{3\times 2^{n+1}}.
$
 It is shown also that an analogous numerical characteristic for numbers of the form $3m+1$ is equal to the number
 $
 \frac{3}{5}+ \frac{(-1)^{n+1}}{15 \times 2^{2(n-1)}}.
$  By using these formulas it is proved that under Collatz conjecture the Collatz mapping has no an asymptotic   mixing  property $\pmod{3}$.
It is constructed also an example of a real-valued function on the cartesian product $\mathbf{N}^2$ of the set of all natural numbers $\mathbf{N}$ such that an equality its repeated integrals ( with respect to Banach measure in $\mathbf{N}$ ) implies that Collatz conjecture fails. In addition, it is demonstrated that Collatz conjecture fails for supernatural numbers.

\end{abstract}



\end{frontmatter}

\section{Introduction}

 There are many works where the authors intensively use a probabilistic approach for the studying  of asymptotic properties of various general systems. For example,
probabilistic proofs of  asymptotic distribution formulas  of the number of short cycles in graphs with a given degree sequence and for
hypergraphs  are given in \cite{Bollobas1980}. By virtue the non-lattice property of the lifetime distribution and the Hewitt-Savage zero-one law,  another probabilistic proof of Blackwell's renewal theorem is given in \cite{Lindvall77} and so on. The present note is like to these article.
More precisely, by virtue properties of Markov  homogeneous chains and Banach measure $P$ in $\mathrm{N}$, we prove that under Collatz conjecture the Collatz mapping has no an asymptotic   mixing  property $\pmod{3}$(This notion is given in Section 5).
 Recall, that the Collatz conjecture is a conjecture in mathematics named after Lothar Collatz, who first proposed it in 1937.
The conjecture is also known as the 3n + 1 conjecture, the Ulam conjecture (after Stanislaw Ulam), Kakutani's problem (after Shizuo Kakutani),
the Thwaites conjecture (after Sir Bryan Thwaites), Hasse's algorithm (after Helmut Hasse), or the Syracuse problem (cf. \cite{r1}).

Take any natural number $n$. If $n$ is even, divide it by $2$ to get $n / 2$. If $n$ is odd, multiply it by $3$ and add $1$ to obtain $3n + 1$.
Repeat the process infinitely. The conjecture is that no matter what number you start with, you will always eventually reach $1$.

As many authors have previously stated, Paul Erd$\acute{o}$s once said, "Mathematics may not be ready for such problems".  Thus
far all evidence indicates he was correct (cf. \cite{r2}).

The conjecture has been checked by computer for all starting values up to $5 × 2^{60} \approx 5.764 \times 10^{18}$ (cf. \cite{r3}).
All initial values tested so far eventually end in the repeating cycle $(4; 2; 1)$, which has only three terms. It is evident that such a
computation result  is not a proof of the Collatz conjecture and it remains today unsolved.

The rest of this note  is the following.

In Section 2 we consider some auxiliary notions and facts from the theory of Markov homogeneous chains and  shift-invariant measure theory in $N$.
In Section 3 we give the formula for a relative frequency
of even numbers in the sequence of $n$-th coordinates of all Collatz sequences. In Section 4 we do same for numbers of the form $3m+1$. In Section 5 we construct a function on the cartesian product $N^2$ of the set of all natural numbers $N$ such that an equality its repeated integrals with respect to Banach measure implies that Collatz conjecture fails.  By using the main result of  Section 4, it is proved (under Collatz conjecture) that Collatz mapping has no an asymptotic   mixing   property $\pmod{3}$. In Section 6 we demonstrate that  Collatz conjecture fails for  supernatural numbers.

\section{Auxiliary notions and facts  from the theory of Markov homogeneous chains and invariant measure theory}

Let $(\Omega,{\mathcal{F}},P)$ be a probability space.
Assume that we have a physical system, which after each step
changes its phase position. Assume that the possible
positions are $\epsilon_0$ and $\epsilon_1$.  Let $\xi_n(\omega)$ be a position of physical system
after $n$ steps $(\omega \in \Omega)$. Clearly, the chain of
logical transitions

\begin{equation}
\xi_0(\omega) \to \xi_1(\omega) \to \cdots ~~(\omega \in \Omega)\label{ccs}
\end{equation} depends on the chance factor. Assume
that the following regularity condition is preserved: if after
~$n$ steps the system is in position ~$\epsilon_i$~, then,
independently of its early positions it will pass to position
$\epsilon_j$ with probability $p_{ij}$, i.e.,

\begin{equation}
p_{ij}=P(\{\omega:\xi_{n+1}(\omega)=\epsilon_j~|\xi_n(\omega)=\epsilon_i\}),\label{ccs}
i,j =0,1.
\end{equation}

The above described model is called Markov homogeneous chain. The following matrix
\[ {\mathcal{P}}=\left(
\begin{array}{cl}
p_{00} & p_{01} \\
p_{10}& p_{11}
\end{array}\right )\]
is called the matrix of transition probabilities of the physical system  after~$1$~ step.

Besides there is also given  the vector of initial distributions of the physical  system , i.e.,

\begin{equation}
(P_0^{(0)},P_1^{(0)}) =(P(\{ \omega:\xi_0(\omega)=\epsilon_0\}),P(\{ \omega:\xi_0(\omega)=\epsilon_1\})).\label{ccs}
\end{equation}

Let denote by $P_i(n)(i=0,1)$ the probability that the physical system will be in the
position ~$\epsilon_i$ after~$n$~ steps, i.e.,

\begin{equation}
P_i(n)=P(\{\omega:\xi_n(\omega)=\epsilon_i\})(i=0,1).\label{ccs}
\end{equation}
Then the vector $(P_0^{(n)},P_1^{(n)})$, defined by
\begin{equation}
(P_0^{(n)},P_1^{(n)}) =(P(\{ \omega:\xi_n(\omega)=\epsilon_0\}),P(\{ \omega:\xi_n(\omega)=\epsilon_1\})),\label{ccs}
\end{equation}
is called the vector of distributions of the physical system after~$n$~ steps.

\begin{lem} (\cite{r4}~ For an arbitrary natural number $n$, the following equality
\begin{equation}
(P_0^{(n)}, P_1^{(n)})=(P_0^{(0)}, P_1^{(0)}) \times \left (
\begin{array}{cl}
p_{00} & p_{01} \\
p_{10}& p_{11}
\end{array}\right )^{n}\label{ccs}
\end{equation}
holds true.
\end{lem}

\medskip

\begin{lem}(\cite{r4})~ If $|p_{00}+p_{11}-1|<1$, then the following equality
\begin{equation}
\lim_{n \to \infty} (P_0^{(n)}, P_1^{(n)})=(\frac{1-p_{11}}{2-p_{00}-p_{11}},\frac{1-p_{00}}{2-p_{00}-p_{11}})\label{ccs}
\end{equation}
holds true.
\end{lem}

\begin{proof} Setting

\begin{equation}
{\mathcal{P}}= \left (
\begin{array}{cl}
p_{00} & p_{01} \\
p_{10}& p_{11}
\end{array}\right ),\label{ccs}
\end{equation}

it is obvious to see that

\begin{equation}
{{\mathcal{P}}}^2= \left(
\begin{array}{cl}
p^2_{00} +p_{01}p_{10}& p_{01}(p_{00}+p_{11}) \\
p_{10}(p_{00}+p_{11})& p^2_{11}+p_{01}p_{10}
\end{array}\right ).\label{ccs}
\end{equation}

By using the method of the mathematical  induction, one can easily show the validity of the following equality

\begin{equation}
{{\mathcal{P}}}^n= \frac{1}{2-p_{00}-p_{11}} \left(
\begin{array}{cl}
1-p_{11}&1-p_{00} \\
1-p_{11}& 1-p_{00}
\end{array}\right )$$
$$+\frac{(p_{00} +p_{11}-1)^n}{2-p_{00}-p_{11}} \left(
\begin{array}{cl}
1-p_{00}&-(1-p_{00}) \\
-(1-p_{11})& 1-p_{11}
\end{array} \right).\label{ccs}
\end{equation}

Since $|p_{00}+p_{11}-1|<1$, we deduce that

\begin{equation}
\lim_{n \to \infty}{{\mathcal{P}}}^n= \frac{1}{2-p_{00}-p_{11}} \left(
\begin{array}{cl}
1-p_{11}&1-p_{00} \\
1-p_{11}& 1-p_{00}
\end{array}\right).\label{ccs}
\end{equation}

Take into account the latter relation, by virtue of Lemma 2.1  we get

\begin{equation*}
\lim_{n \to \infty} (P_0^{(n)}, P_1^{(n)})=\lim_{n \to \infty}(P_0^{(0)}, P_1^{(0)}) \times \left (
\begin{array}{cl}
p_{00} & p_{01} \\
p_{10}& p_{11}
\end{array}\right )^{n}=
\end{equation*}

\begin{equation*}
=(P_0^{(0)}, P_1^{(0)}) \times \lim_{n \to \infty}\left (
\begin{array}{cl}
p_{00} & p_{01} \\
p_{10}& p_{11}
\end{array}\right )^{n}=
\end{equation*}

\begin{equation*}
(P_0^{(0)}, P_1^{(0)}) \times \frac{1}{2-p_{00}-p_{11}} \left (
\begin{array}{cl}
1-p_{11}&1-p_{00} \\
1-p_{11}& 1-p_{00}
\end{array}\right )=
\end{equation*}

\begin{equation}
=(\frac{1-p_{11}}{2-p_{00}-p_{11}},\frac{1-p_{00}}{2-p_{00}-p_{11}}). \label{ccs}
\end{equation}
This ends the proof of Lemma 2.2

\end{proof}

\begin{defn} Let $\mathcal{A}$ be a shift invariant algebra of subsets of  the set of all natural numbers $\mathbf{N}$.
A measure $P$ on $\mathcal{A}$ is said to be shift-invariant (also called translation-invariant) if
the condition
\begin{equation}
P(X+1)=P(X) \label{ccs}
\end{equation}
holds true for $X \in \mathcal{A}$ , where $X+1=\{ x+1:x \in X\}$.
\end{defn}

\begin{defn}
We say that $ K \subseteq  \mathbf{N}$ has  the (asymptotic) density  $d(K)$ if the limit
\begin{equation}
 d(K)=\lim_{n \to \infty} \#([1,n]\cap K)/n  \label{ccs}
 \end{equation}
 exists, where $\#(\cdot)$ denotes the counting measure.
\end{defn}

\begin{lem}(\cite{Banach1932}, p. 341) There exists  a shift-invariant
measure $P$ defined on the powerset of $N$ for which $P(\mathbf{N})=1$.
\end{lem}

\begin{lem}(\cite{Bumby1969}, Examples, p. 371; p. 401 ) If $P$ is a shift-invariant
measure defined on the powerset of $\mathbf{N}$ for which  $P(\mathbf{N})=1$, then
\begin{equation}
P(\{ p + r \mathbf{N} \}) = r^{-1} \label{ccs}
\end{equation}
for all $p \in  [0,\infty)$ and $r \in [1,\infty)$.
\end{lem}

\begin{lem} Let  $P$ be a shift-invariant probability measure defined on the powerset of $\mathbf{N}$. Let $f: \mathbf{N}\to {\bf R}$ be such a mapping that
 there exists a finite limit $\lim_{n \to \infty}f(n)=A$. Then $\int_{\mathbf{N}}f(n)dP(n)=A$.
\end{lem}
\begin{proof} Let $\epsilon>0$. Then there is $n(\epsilon)>0$ such that $|f(n)-A|<\epsilon$ for $n>n(\epsilon)$.
\begin{equation*}
\int_{\mathbf{N}}f(n)dP(n)=\int_{[0,n(\epsilon)[ }f(n)dP(n)+\int_{\mathbf{N}\setminus [0,n(\epsilon)[}f(n)dP(n)=
\end{equation*}
\begin{equation}
\int_{\mathbf{N}\setminus [0,n(\epsilon)[}f(n)dP(n).   \label{ccs}
\end{equation}
Since $A-\epsilon< f(n)< A-\epsilon$ for $n>n(\epsilon)$, we get that
\begin{equation*}
A-\epsilon =  \int_{\mathbf{N}\setminus [0,n(\epsilon)[}(A-\epsilon)dP(n) <  \int_{\mathbf{N}\setminus [0,n(\epsilon)[}f(n)dP(n)<
\end{equation*}
\begin{equation}
  \int_{\mathbf{N}\setminus [0,n(\epsilon)[}(A+\epsilon)dP(n)=A+\epsilon.   \label{ccs}
\end{equation}
By formula (2.16)  we get
\begin{equation}
A-\epsilon <  \int_{\mathbf{N}}f(n)dP(n) < A+\epsilon. \label{ccs}
\end{equation}

Since $\epsilon$ was taken arbitrary, Lemma 2.5 is proved.

\end{proof}

\begin{lem} Let  $P$ be a shift-invariant measure defined on the powerset of $\mathbf{N}$.  Then for all $p \in  [0, \infty)$ and $r \in  [ 1,\infty)$
a set $K= p + r N$ has the density $d(K)$ and the following equality
$d(K)=P(K)$ holds true.
\end{lem}

In the sequel,  under  a probability space  $(\Omega,{\mathcal{F}},P)$~  we will assume a triplet  $({\mathbf{N}}, {\mathcal{P}}({\mathbf{N}}), P)$,
where ${\mathcal{P}}({\mathbf{N}})$ denotes the powerset of the set of all natural numbers $\mathbf{N}$ and  the measure $P$ comes from Lemma 2.3.

In modular arithmetic notations, define the random variable $\xi$ as follows:

\begin{equation}
\xi(\omega) = \begin{cases} \omega/2 &\text{if } \omega \equiv 0 \pmod{2}\\ 3\omega+1 & \text{if } \omega \equiv 1 \pmod{2} .\end{cases}\label{ccs}
\end{equation}

\section{On a formula for  relative frequencies of even and odd numbers in the sequence of $n$-th coordinates of all Collatz sequences }

By virtue the results of Lemmas 2.3-2.4, we get the validity of the following assertion.

\begin{lem}The following equalities hold true:
\begin{equation}
p_{00}:=P(\{ \omega :\xi(\omega)\equiv 0 \pmod{2}| \omega\equiv 0 \pmod{2}\})=1/2,\label{ccs}
\end{equation}
\begin{equation}
p_{01}:=P(\{ \omega :\xi(\omega)\equiv 1 \pmod{2}| \omega\equiv 0 \pmod{2} \})=1/2,\label{ccs}
\end{equation}
\begin{equation}
p_{10}:=P(\{ \omega :\xi(\omega)\equiv 0 \pmod{2}| \omega\equiv 1 \pmod{2}\})=1,\label{ccs}
\end{equation}
\begin{equation}
p_{11}:=P(\{ \omega :\xi(\omega)\equiv 1 \pmod{2}| \omega  \equiv 1 \pmod{2}\})=0.\label{ccs}
\end{equation}
\end{lem}

\begin{proof} {\bf Proof of the equality (3.1)}
\begin{equation*}
p_{00}:=P(\{ \omega :\xi(\omega)\equiv 0 \pmod{2}| \omega\equiv 0 \pmod{2}\})=
\end{equation*}
\begin{equation*}
P(\{ \omega : \xi(\omega)\equiv 0 \pmod{2}\} \cap \{ \omega : \omega  \equiv 0 \pmod{2}\})/ P( \{ \omega : \omega  \equiv 0 \pmod{2}\})=
\end{equation*}
\begin{equation}
P((4N\cup(2\mathbf{N}+1))\cap 2\mathbf{N})/P(2\mathbf{N})= P(4\mathbf{N})/ P(2\mathbf{N})=1/4:1/2=1/2. \label{ccs}
\end{equation}
{\bf Proof of the equality (3.2).}
\begin{equation*}
p_{01}:=P(\{ \omega :\xi(\omega)\equiv 1 \pmod{2}| \omega\equiv 0 \pmod{2} \})=
\end{equation*}
\begin{equation*}
P( \{ \omega :\xi(\omega)\equiv 1 \pmod{2} \} \cap \{ \omega :  \omega\equiv 0 \pmod{2}\}) /  P( \{ \omega : \omega\equiv 0 \pmod{2}\})=
\end{equation*}
\begin{equation}
P(4\mathbf{N} +2)/ P(2\mathbf{N}) =1/4 : 1/2 =1/2.  \label{ccs}
\end{equation}
{\bf Proof of the equality (3.3.)}
\begin{equation*}
p_{10}:=P(\{ \omega :\xi(\omega)\equiv 0 \pmod{2}| \omega\equiv 1 \pmod{2}\})=
\end{equation*}
\begin{equation*}
P(\{ \omega :\xi(\omega)\equiv 0 \pmod{2}\} \cap \{ \omega : \omega \equiv 1 \pmod{2}\})/ P(\{ \omega : \omega \equiv 1 \pmod{2}\})
\end{equation*}
\begin{equation}
P((2N+1 \cup  4\mathbf{N} ) \cap 2\mathbf{N}+1 ) / P(2\mathbf{N}+1)=P(2\mathbf{N}+1) / P(2\mathbf{N}+1)= 1.    \label{ccs}
\end{equation}
{\bf Proof of the equality (3.4.)}
\begin{equation*}
P(\{ \omega :\xi(\omega)\equiv 1 \pmod{2}| \omega  \equiv 1 \pmod{2}\})=
\end{equation*}
\begin{equation*}
P(\{ \omega :\xi(\omega)\equiv 1 \pmod{2}\} \cap \{ \omega : \omega  \equiv 1 \pmod{2}\}) / P(\{ \omega : \omega  \equiv 1 \pmod{2}\})     =
\end{equation*}
\begin{equation}
P(4N \cap (2\mathbf{N}+1))/P(2\mathbf{N}+1)= P( \emptyset )/P(2\mathbf{N}+1)=0.\label{ccs}
\end{equation}

\end{proof}

Let $\xi_n(\omega)$ be  $\xi$  applied to $\omega$  recursively $n$ times;  Clearly, the chain of
logical transitions
\begin{equation}
\xi_0(\omega) \to \xi_1(\omega) \to \cdots ~~(\omega \in \Omega)\label{ccs}
\end{equation}
depends on the chance factor, where $\xi_0(\omega)=\omega$ and  $\xi_1=\xi$ for each $\omega \in \Omega$.

The sequence $(\xi_{n}(\omega))_{n \in {\mathbf{N}}}$ is called a Collatz sequence starting at point $\omega$.

We say that a system $\xi_n(\omega)$ is in a position $\epsilon_0$ (and write $\xi_n(\omega) \in \epsilon_0$) if $\xi_n(\omega)\equiv 0 \pmod{2}$
and $\xi_n(\omega)$ is in a position $\epsilon_1$ (and write $\xi_n(\omega) \in \epsilon_1$) if  $\xi_n(\omega)\equiv 1\pmod{2}$. Hence we have
a system, which after each step
changes its  position and a set of possible
positions are $\{\epsilon_0,\epsilon_1 \}$.

By the method of mathematical induction one can get the validity of the following assertion.

\begin{lem} For each natural number $n$, a set $\{ \omega \in \Omega : \xi_n(\omega) \in \epsilon_0\}$ can be presented  as  a finite union of disjoint sets of the form
$p + r \mathbf{N}$ where $p \in  [0, \infty)$ and $r \in  [ 1,\infty)$.
\end{lem}

If after
~$n$ steps the system is in position ~$\epsilon_i$~, then,
independently of its early positions it will pass to position
$\epsilon_j$ with probability $p_{ij}$, i.e.,
\begin{equation}
p_{ij}=P(\{\omega:\xi_{n+1}(\omega) \in  \epsilon_j~|\xi_n(\omega) \in  \epsilon_i\}),\label{ccs}
i,j =0,1.
\end{equation}

The above described model is just  Markov  homogeneous chain  with the matrix of transition probabilities
\begin{equation}
{{\mathcal{P}}}=\left(
\begin{array}{cl}
1/2 & 1/2 \\
1& 0
\end{array}\right )\label{ccs}
\end{equation}

and   with the vector of initial distributions

\begin{equation}
(P_0^{(0)},P_1^{(0)}) =(1/2,1/2).\label{ccs}
\end{equation}

Setting
\begin{equation}
(P_0^{(n)}, P_1^{(n)})= ( P(\{\omega:\xi_n(\omega) \in \epsilon_0\}), P(\{\omega:\xi_n(\omega) \in \epsilon_1\})),\label{ccs}
\end{equation}

by Lemma 2.1  we get

\begin{equation}
\lim_{n \to \infty}(P_0^{(n)}, P_1^{(n)})=(1/2,1/2) \times \left (
\begin{array}{cl}
\frac{1}{2} & \frac{1}{2} \\
1& 0
\end{array}\right )^{n}. \label{ccs}
\end{equation}

For each natural number $n \in \mathbf{N}$, let denote by $\nu^{(n)}(\epsilon_0)$ and $\nu^{(n)}(\epsilon_1)$  relative frequencies  of even and odd numbers within the sequence $(\xi_n(k))_{ k  \in {\mathbf{N}}}$ of $n$-th coordinates of all Collatz sequences, respectively, i.e.,

\begin{equation}
  \nu^{(n)}(\epsilon_0)=\lim_{m \to \infty} \#\{ k: 1 \le k \le m~\&~\xi_n(k)\in \epsilon_0\}/m \label{ccs}
 \end{equation}
and
\begin{equation}
  \nu^{(n)}(\epsilon_1)=\lim_{m \to \infty} \#\{ k: 1 \le k \le m~\&~\xi_n(k)\in \epsilon_1\}/m ). \label{ccs}
 \end{equation}

\begin{thm} For each natural number $n \in \mathbf{N}$, the  relative frequencies  $\nu^{(n)}(\epsilon_0)$ and $\nu^{(n)}(\epsilon_1)$ are computed by the following formula
 \begin{equation}
  \nu^{(n)}(\epsilon_0)= \frac{2}{3}+(-1)^{n+1}\times \frac{1}{3 \times 2^{n+1}} \label{ccs}
\end{equation}
and
\begin{equation}
  \nu^{(n)}(\epsilon_1)= \frac{1}{3}+(-1)^{n}\times \frac{1}{3 \times 2^{n+1}}. \label{ccs}
\end{equation}
\end{thm}

\begin{proof}

By using the approach used in the proof of the Lemma 2.2, we get

\begin{equation*}
\left (
\begin{array}{cl}
\frac{1}{2} & \frac{1}{2} \\
1& 0
\end{array}\right )^{n}= \frac{2}{3} \left(
\begin{array}{cl}
1&1/2 \\
1& 1/2
\end{array}\right )+\frac{(-1)^n}{3 \times 2^{n-1}} \left(
\begin{array}{cl}
1/2&-1/2) \\
-1 & 1
\end{array} \right)=
\end{equation*}

\begin{equation}
\left (
\begin{array}{cl}
\frac{2}{3}+\frac{(-1)^n}{3\times 2^{n}} & \frac{1}{3}+\frac{(-1)^{n+1}}{3\times 2^{n}} \\
\frac{2}{3}+\frac{(-1)^{n+1}}{3\times 2^{n-1}}& \frac{1}{3}+\frac{(-1)^{n}}{3\times 2^{n-1}}
\end{array}\right ).\label{ccs}
\end{equation}

On the one hand, for each natural number $n \in \mathbf{N}$, the relative frequency  of even  numbers within the sequence $(\xi_n(k))_{ k  \in {\mathbf{N}}}$ of $n$-th coordinates of all Collatz sequences coincides with  the relative frequency  of the pre-image $\xi_n^{-1}(2\mathbf{N})$ within $\Omega=N$. By Lemma 2.7, the pre-image $\xi_n^{-1}(2\mathbf{N})$ is  presented as  a finite union of disjoint sets of the form $p + r \mathbf{N}$ where $p \in  [0, \infty)$ and $r \in  [ 1,\infty)$. By Lemma 2.6,  such sets have  relative frequencies(equivalently, density) which coincide with their  Banach measure $P$. Take into account this fact and the formula (2.6), we get

\begin{equation*}
 (\nu^{(n)}(\epsilon_0), \nu^{(n)}(\epsilon_1)=( P(\{\omega:\xi_n(\omega) \in \epsilon_0\}), P(\{\omega:\xi_n(\omega) \in \epsilon_1\}))=
 \end{equation*}
 \begin{equation*}
 (P_0^{(n)}, P_1^{(n)})=(1/2,1/2) \times \left (
\begin{array}{cl}
\frac{1}{2} & \frac{1}{2} \\
1& 0
\end{array}\right )^{n}=
\end{equation*}

\begin{equation*} (1/2,1/2) \times \left (
\begin{array}{cl}
\frac{2}{3}+\frac{(-1)^n}{3\times 2^{n}} & \frac{1}{3}+\frac{(-1)^{n+1}}{3\times 2^{n}} \\
\frac{2}{3}+\frac{(-1)^{n+1}}{3\times 2^{n-1}}& \frac{1}{3}+\frac{(-1)^{n}}{3\times 2^{n-1}}
\end{array}\right )=
\end{equation*}

\begin{equation}
\big( \frac{2}{3}+\frac{(-1)^{n+1}}{3\times 2^{n+1}}, \frac{1}{3}+\frac{(-1)^n}{3\times 2^{n+1}} \big ).\label{ccs}
\end{equation}

This ends the proof of Theorem 3.1.
\end{proof}

\begin{cor} The following equalities
 \begin{equation}
  \lim_{n \to \infty} \nu^{(n)}(\epsilon_0)= 2/3\label{ccs}
 \end{equation}
 and
 \begin{equation}
  \lim_{n \to \infty}\nu^{(n)}(\epsilon_1)= 1/3.\label{ccs}
 \end{equation}
hold true.
\end{cor}

\section{On a formula for a relative frequency of numbers of the type $3m+1$ in the sequence of $n$-th coordinates of all Collatz sequences }

\begin{lem}The following equalities hold true:
\begin{equation}
p_{00}:=P(\{ \omega :\xi(\omega)= 1 \pmod{3}| \omega = 1 \pmod{3}\})=1/2,\label{ccs}
\end{equation}
\begin{equation}
p_{01}:=P(\{ \omega :\xi(\omega) \neq 1 \pmod{3}| \omega = 1 \pmod{3} \})=1/2,\label{ccs}
\end{equation}
\begin{equation}
p_{10}:=P(\{ \omega :\xi(\omega)= 1 \pmod{3}| \omega \neq 1 \pmod{3}\})=3/4,\label{ccs}
\end{equation}
\begin{equation}
p_{11}:=P(\{ \omega :\xi(\omega) \neq 1 \pmod{3}| \omega  \neq 1 \pmod{3}\})=1/4.\label{ccs}
\end{equation}
\end{lem}

\begin{proof}
We have:
\begin{equation*}
\xi(3N)=\xi(3((2N+1) \cup 2N)))=\xi(3((2N+1)) \cup 3\times 2N) =
\end{equation*}
\begin{equation}
\xi(6N+3) \cup \xi(2 \times 3N)= (3(6N+3)+1) \cup 3N;  \label{ccs}
\end{equation}
\begin{equation*}
\xi(3N+1)=\xi(3((2N+1) \cup 2N))+1 )=\xi(3((2N+1))+1 \cup 3\times 2N+1) =
\end{equation*}
\begin{equation}\xi(6N+4) \cup \xi(6N+1)=(3N+2) \cup (3(6N+1)+1); \label{ccs}
\end{equation}
\begin{equation*}
\xi(3N+2)=\xi(3((2N+1) \cup 2N))+2 )=\xi(3((2N+1))+2 \cup 3\times 2N+2) =
\end{equation*}
\begin{equation}\xi(6N+5) \cup \xi(6N+2)=(3(6N+2)+1)\cup (3N+1). \label{ccs}
\end{equation}
By using formulas (4.5)-(4.7) and Lemma 2.4, we get:

\begin{equation*}
p_{00}:=P(\{ \omega :\xi(\omega)= 1\pmod{3}\big|\omega = 1\pmod{3}\})=
\end{equation*}
\begin{equation*}
P(\{\omega :\xi(\omega)=1\pmod{3}~\&~\omega = 1\pmod{3}\})/P(\{ \omega : \omega = 1\pmod{3}\})=
\end{equation*}
\begin{equation*}
P(((6N+3) \cup (6N+1) \cup (3N+2) )\cap (3N+1) )/P(3N+1)=
\end{equation*}
\begin{equation}P((6N+1)\cap (3N+1) )/P(3N+1)=1/6:1/3=1/2;\label{ccs}
\end{equation}
\begin{equation*}
p_{11}:=P(\{ \omega :\xi(\omega)\neq 1\pmod{3}\big| \omega \neq 1 \pmod{3}\})=
\end{equation*}
\begin{equation*}
 P(\{ \omega :\xi(\omega)\neq 1 \pmod{3}~\&~ \omega \neq 1 \pmod{3}\})/ P(\{ \omega :  \omega \neq 1 \pmod{3}\})=
 \end{equation*}
\begin{equation*} P((6N \cup (6N+4)) \cap (3N \cup 3N+2))/ P((3N \cup (3N+2))=
\end{equation*}
\begin{equation}P(6N)/ P(3N \cup (3N+2))=1/6:2/3=1/4.
\label{ccs}
\end{equation}
It is clear that $p_{01}=1-p_{00}=1/2$ and $p_{10}=1-p_{11}=3/4$.

\end{proof}

We say that a Collatz sequence $\xi_n(\omega)$ starting at point $\omega$ is in a position $\epsilon_0$ (and write $\xi_n(\omega) \in \epsilon_0$) if $\xi_n(\omega)\equiv 1 \pmod{3}$
and $\xi_n(\omega)$ is in a position $\epsilon_1$, otherwise.

By the method of mathematical induction one can get the validity of the following assertion.

\begin{lem} For each natural number $n$, a set $\{ \omega \in \Omega : \xi_n(\omega) \in \epsilon_0\}$ can be presented  as  a finite union of disjoint sets of the form
$p + r \mathbf{N}$ where $p \in  [0, \infty)$ and $r \in  [ 1,\infty)$.
\end{lem}

If after
~$n$ steps the system is in position ~$\epsilon_i$~, then,
independently of its early positions it will pass to position
$\epsilon_j$ with probability $p_{ij}$, i.e.,
\begin{equation}
p_{ij}=P(\{\omega:\xi_{n+1}(\omega) \in  \epsilon_j~|\xi_n(\omega) \in  \epsilon_i\}),\label{ccs}
i,j =0,1.
\end{equation}

The above described model is just  Markov  homogeneous chain  with the matrix of transition probabilities
\begin{equation}
{{\mathcal{P}}}=\left(
\begin{array}{cl}
1/2 & 1/2 \\
3/4& 1/4
\end{array}\right )\label{ccs}
\end{equation}

and   with the vector of initial distributions

\begin{equation}
(P_0^{(0)},P_1^{(0)}) =(1/3,2/3).\label{ccs}
\end{equation}

Setting
\begin{equation}
(P_0^{(n)}, P_1^{(n)})= ( P(\{\omega:\xi_n(\omega) \in \epsilon_0\}), P(\{\omega:\xi_n(\omega) \in \epsilon_1\})),\label{ccs}
\end{equation}

by Lemma 2.1  we get

\begin{equation}
(P_0^{(n)}, P_1^{(n)})=(1/3,2/3) \times \left (
\begin{array}{cl}
\frac{1}{2} & \frac{1}{2} \\
\frac{3}{4} & \frac{1}{4}
\end{array}\right )^{n}. \label{ccs}
\end{equation}

For each natural number $n \in \mathbf{N}$, let define $\nu^{(n)}(\epsilon_0)$ and $\nu^{(n)}(\epsilon_1)$ by the formulas

\begin{equation}
  \nu^{(n)}(\epsilon_0)=\lim_{m \to \infty} \#\{ k: 1 \le k \le m~\&~\xi_n(k) \in \epsilon_0\}/m \label{ccs}
 \end{equation}
and
\begin{equation}
  \nu^{(n)}(\epsilon_1)=\lim_{m \to \infty} \#\{ k: 1 \le k \le m~\&~\xi_n(k)\in \epsilon_1\}/m ). \label{ccs}
 \end{equation}

\begin{thm} For each natural number $n \in \mathbf{N}$, the   following formulas
 \begin{equation}
  \nu^{(n)}(\epsilon_0)= \frac{3}{5}+ \frac{(-1)^{n+1}}{15 \times 2^{2(n-1)}}\label{ccs}
\end{equation}
and
\begin{equation}
  \nu^{(n)}(\epsilon_1)= \frac{2}{5}+ \frac{(-1)^{n}}{15 \times 2^{2(n-1)}} \label{ccs}
\end{equation}
hold true.

\end{thm}

\begin{proof}

By using the approach used in the proof of Lemma 2.2, we get

\begin{equation*}
\left (
\begin{array}{cl}
\frac{1}{2} & \frac{1}{2} \\
3/4& 1/4
\end{array}\right )^{n}= \frac{4}{5} \left(
\begin{array}{cl}
3/4&1/2 \\
3/4& 1/2
\end{array}\right )+\frac{(-1)^n}{5 \times 2^{2n-2}} \left(
\begin{array}{cl}
1/2&-1/2) \\
-3/4 & 3/4
\end{array} \right)=
\end{equation*}

\begin{equation*}
 \left(
\begin{array}{cl}
3/5&2/5 \\
3/5& 2/5
\end{array}\right )+ \left(
\begin{array}{cl}
\frac{(-1)^n}{5 \times 2^{2n-1}}&\frac{(-1)^{n+1}}{5 \times 2^{2n-1}}) \\
\frac{3(-1)^{n+1}}{5 \times 2^{2n}} & \frac{3(-1)^n}{5 \times 2^{2n}}
\end{array} \right)=
\end{equation*}

\begin{equation*}
 \left(
\begin{array}{cl}
3/5+\frac{(-1)^n}{5 \times 2^{2n-1}}&2/5+\frac{(-1)^{n+1}}{5 \times 2^{2n-1}}\\
3/5+\frac{3(-1)^{n+1}}{5 \times 2^{2n}}& 2/5+\frac{3(-1)^n}{5 \times 2^{2n}}
\end{array}\right ).
\end{equation*}

Notice that for each natural number $n \in \mathbf{N}$, the relative frequency  of   numbers within the sequence $(\xi_n(k))_{ k  \in {\mathbf{N}}}$ of $n$-th coordinates of all Collatz sequences which are presented in the form $3m+1$ coincides with  the relative frequency  of the pre-image $\xi_n^{-1}(3\mathbf{N}+1)$ within $\Omega=N$. By Lemma 4.2, the pre-image $\xi_n^{-1}(3\mathbf{N}+1)$ is  presented as  a finite union of disjoint sets of the form $p + r \mathbf{N}$ where $p \in  [0, \infty)$ and $r \in  [ 1,\infty)$. By Lemma 2.5,  such sets have  relative frequencies which coincide with their  Banach measure $P$. Take into account this fact and the formula 2.31, we get

\begin{equation*}
 \big(\nu^{(n)}(\epsilon_0), \nu^{(n)}(\epsilon_1)\big)= \big( P(\{\omega:\xi_n(\omega) \in \epsilon_0\}), P(\{\omega:\xi_n(\omega) \in \epsilon_1\}) \big)=
 \end{equation*}
 \begin{equation*}
  \big(P_0^{(n)}, P_1^{(n)} \big)=(1/3,2/3) \times \left (
\begin{array}{cl}
\frac{1}{2} & \frac{1}{2} \\
3/4& 1/4
\end{array}\right )^{n}=
\end{equation*}

\begin{equation*} (1/3,2/3) \times \left(
\begin{array}{cl}
3/5+\frac{(-1)^n}{5 \times 2^{2n-1}}&2/5+\frac{(-1)^{n+1}}{5 \times 2^{2n-1}}\\
3/5+\frac{3(-1)^{n+1}}{5 \times 2^{2n}}& 2/5+\frac{3(-1)^n}{5 \times 2^{2n}}
\end{array}\right )=
\end{equation*}

\begin{equation}
\big( \frac{3}{5}+ \frac{(-1)^{n+1}}{15 \times 2^{2(n-1)}}, \frac{2}{5}+\frac{(-1)^n}{15\times 2^{2(n-1)}} \big ).\label{ccs}
\end{equation}

This ends the proof of Theorem 4.1.
\end{proof}

\begin{cor} The following equalities
 \begin{equation}
  \lim_{n \to \infty} \nu^{(n)}(\epsilon_0)= 3/5\label{ccs}
 \end{equation}
 and
 \begin{equation}
  \lim_{n \to \infty}\nu^{(n)}(\epsilon_1)= 2/5.\label{ccs}
 \end{equation}
hold true.
\end{cor}

\section{Main results}
Let denote by
 \begin{equation*}
\mathcal{G}=\{ f | f : \Omega \times \Omega \to {\bf R} ~\&~ \int_{\Omega}\big(\int_{\Omega}f(\omega_1,\omega_2)dP(\omega_2)\big)d P(\omega_1)=
\end{equation*}
 \begin{equation}
\int_{\Omega}\big(\int_{\Omega}f(\omega_1,\omega_2)dP(\omega_1)\big)d P(\omega_2)\}\label{ccs}
 \end{equation}

\begin{thm} Let define $g : \Omega \times \Omega \to {\bf R}$ by the following formula: $g(\omega_1,\omega_2)=1$ if $\xi_{\omega_1}(\omega_2)=0 \pmod{2}$ and $g(\omega_1,\omega_2)=0$, otherwise. Then under Collatz conjecture we have that  $g \in \mathcal{G}$.
\end{thm}
\begin{proof}
By Lemma 2.5 and Theorem 3.1, we get
\begin{equation*}
\int_{\Omega}\big(\int_{\Omega}g(\omega_1,\omega_2)dP(\omega_2)\big)d P(\omega_1)=
\end{equation*}
  \begin{equation}
\int_{\Omega}\big( \frac{2}{3}+\frac{(-1)^{\omega_1+1}}{3\times 2^{\omega_1+1}}\big) d P(\omega_1)=\frac{2}{3}.\label{ccs}
\end{equation}
If Collatz Conjecture is valid then relative frequency of even numbers within each Collatz sequence will be equal to $2/3$ which implies

 \begin{equation}
\int_{\Omega}\big(\int_{\Omega}g(\omega_1,\omega_2)dP(\omega_1)\big)d P(\omega_2)=\int_{\Omega}2/3 d P(\omega_2)=\frac{2}{3}.\label{ccs}
\end{equation}

This ends the proof of Theorem 5.1.
\end{proof}

\begin{thm} Let define $h : \Omega \times \Omega \to {\bf R}$ by the following formula: $h(\omega_1,\omega_2)=1$ if $\xi_{\omega_1}(\omega_2)=1 \pmod{3}$ and $h(\omega_1,\omega_2)=0$, otherwise. Under Collatz conjecture, we have that  $h \notin \mathcal{G}$.
\end{thm}

\begin{proof}  Under Collatz conjecture  we get
\begin{equation}
\int_{\Omega}\big(\int_{\Omega}h(\omega_1,\omega_2)dP(\omega_1)\big)d P(\omega_2)=\int_{\Omega} 2/3 d P(\omega_2)=2/3.    \label{ccs}
 \end{equation}
On the other hand, by  Lemmas 2.5 and Theorem 4.1, we get
\begin{equation}
 \int_{\Omega}\big(\int_{\Omega}f(\omega_1,\omega_2)dP(\omega_2)\big)d P(\omega_1)=\int_{\Omega}\big(\frac{3}{5}+ \frac{(-1)^{\omega_1+1}}{15 \times 2^{2(\omega_1-1)}} \big) dP(\omega_1)=3/5. \label{ccs}
 \end{equation}

This means that  $h \notin \mathcal{G}$ and Theorem 5.2 is proved.
\end{proof}

\begin{cor}If $h \in \mathcal{G}$ then Collatz conjecture fails.

\end{cor}

Let ${\bf ZFC}$  denotes the Zermelo-Fraenkel set theory with Axiom of Choice(see, for example \cite{Jech78}).   In context with Theorem 5.2 and Corollary 5.1 we state the following

\begin{prob} Is the sentence ''The function $h$ belongs to the class $\mathcal{G}$" consistent with ${\bf ZFC}$ ?
\end{prob}

\begin{defn}We say that a  mapping  $ T : \Omega  \to \Omega $ has  an asymptotic  mixing  property $\pmod{3}$ if there exists an asymptotic sequence
\begin{equation}
\big(\lim_{m \to \infty}T^{3(m-1)+i}(k)\big)_{k \in \mathbf{N}}    \label{ccs}
 \end{equation}
 for each $i(i=0,1,2)$  and the following two conditions hold:

 (i) A relative frequency of numbers of the type $3m+1$ in each
asymptotic sequence is equal to the limit of the sequence of  relative frequencies of numbers of the form $3m+1$ in the
sequence $(T^n(\omega))_{\omega  \in  \Omega}$ of $n$-th coordinates of all sequences $((T^n(\omega))_{n \in N})_{\omega  \in  \Omega}$ when $n$ tends to $\infty$;

  (ii) A relative frequency of even numbers  in each
asymptotic sequence is equal to the limit of the sequence of  relative frequencies of even numbers  in the
sequence $(T^n(\omega))_{\omega  \in  \Omega}$ of $n$-th coordinates of all sequences $((T^n(\omega))_{n \in N})_{\omega  \in  \Omega}$  when $n$ tends to $\infty$.

 \end{defn}

\begin{thm} Under Collatz conjecture, the Collatz mapping has no an asymptotic   mixing   property$\pmod{3}$.
\end{thm}
\begin{proof} Assume the contrary and let  Collatz mapping has  an asymptotic   mixing$\pmod{3}$   property.
For each $k \in \mathbf{N}$, let look at each Collatz sequence $(\xi_n(k))_{ n  \in {\mathbf{N}}}$ as an element $(\xi_{3(m-1)}(k),\xi_{3(m-1)+1}(k),\xi_{3(m-1)+2}(k))_{m \in \mathbf{N}}$  of the $(\mathbf{N}^3)^{\mathbf{N}}$. For example, \begin{equation}
 (4,2,1,4,2,1 \cdots)=((4,2,1),(4,2,1), \cdots)\big).
\label{ccs}
 \end{equation}
  Notice that under Collatz conjecture,  for each fixed $k \in N$, the sequence $(\xi_{3(m-1)}(k),\xi_{3(m-1)+1}(k),\xi_{3(m-1)+2}(k))_{m \in \mathbf{N}}$ must convergent in $N^3$.
If we  equip the space $(\mathbf{N}^3)^\mathbf{N}$ with  Tychonoff metric $\rho$, defined by
 \begin{equation}
\rho((x_n)_{n \in N},(y_n)_{n \in \mathbf{N}})=\sum_{n \in N}\frac{||x_n-y_n||_3}{2^{n+1}(1+||x_n-y_n||_3)}\label{ccs}
 \end{equation}
for each $(x_n)_{n \in \mathbf{N}},(y_n)_{n \in \mathbf{N}} \in (\mathbf{N}^3)^\mathbf{N}$ (here  $||\cdot||_3$ denotes a usual norm in $R^3$),
 then under  Collatz conjecture  the sequence of $m$-th  three-dimensional coordinates of all Collatz sequence
 \begin{equation}
\big( \big(\xi_{3(m-1)}(k),\xi_{3(m-1)+1}(k),\xi_{3(m-1)+2}(k)\big)_{k \in N}\big)_{m \in \mathbf{N}} \label{ccs}
 \end{equation}
must  converge in $(\mathbf{N}^3)^\mathbf{N}$ with respect to Tychonoff metric. Since  the limit of  three-dimensional coordinates of each Collatz sequence
must belong to the set $\{(4,2,1),(2,1,4),(1,4,2)\}$, we deduce that the limit of the sequence (5.9) when $m$ tends to infinity
 must  belong to the set
 \begin{equation}\{(4,2,1),(2,1,4),(1,4,2)\}^\mathbf{N}.
 \label{ccs}
 \end{equation}
    For $i =0,1,2,$ let denote by $\nu_i(4),\nu_i(2), \nu_i(1)$ the relative frequencies of numbers $4,2,1$ in $i$th Collatz asymptotic sequences, respectively.
    Under Collatz conjecture and by the assumption that the  numbers $1,4,2$  have relative frequencies in each Collatz asymptotic sequences(this follows from the asymptotic   mixing   property$\pmod{3}$ of the Collatz mapping) we get
    \begin{equation}\nu_0(4)=\nu_1(2)=\nu_2(1),~  \nu_0(2)=\nu_1(1)=\nu_2(4),~ \nu_0(1)=\nu_1(4)=\nu_2(2).
 \label{ccs}
 \end{equation}

    On the one hand, by  Corollary 3.1 we know  that the limit of the sequence of  relative frequencies of odd numbers  in the
sequence $(T^n(\omega))_{\omega  \in  \Omega}$ of $n$-th coordinates of all sequences $((T^n(\omega))_{n \in N})_{\omega  \in  \Omega}$  when $n$ tends to $\infty$ is equal to $1/3$. Taking into account this fact and the asymptotic   mixing   property$\pmod{3}$ of the Collatz mapping we deduce that $\nu_0(1)=\nu_1(1)=\nu_2(1)=1/3$. The latter equations and formula (5.11)  imply that
   \begin{equation}\nu_0(4)=\nu_0(2)=\nu_0(1)=\nu_1(4)=\nu_1(2)=\nu_1(1)=\nu_2(4)=\nu_2(2)=\nu_2(1)=1/3.
 \label{ccs}
 \end{equation}

    On the other hand, by Corollary 4.1 we know  that  the limit of the sequence of  relative frequencies of numbers having  a form $3m+1$  in the
sequence $(T^n(\omega))_{\omega  \in  \Omega}$ of $n$-th coordinates of all sequences $((T^n(\omega))_{n \in N})_{\omega  \in  \Omega}$  when $n$ tends to $\infty$ is equal to $3/5$. Now taking  into account this fact and  the asymptotic   mixing   property$\pmod{3}$ of the Collatz mapping we deduce that $\nu_0(\{ 1, 4 \})=\nu_1(\{ 1, 4 \})=\nu_2(\{ 1, 4 \})=3/5$. The latter equalities fail  because
   \begin{equation}3/5=\nu_0(\{ 1, 4 \})\neq \nu_0(1)+\nu_0(4)=2/3.
 \label{ccs}
 \end{equation}

  This ends the proof of the theorem.

\end{proof}

\section{Collatz conjecture for supernatural numbers}

The set $N_S$ of supernatural (Steinitz) numbers (see, for example, \cite{Ribes2000}, \cite{Wilson1998}, \cite{Vourdas2013}) is:
$$
N_S =\{ n =\prod_{p \in \prod} p^{e_p(n)} | e_p(n) \in \mathbf{Z}_0^+\cup \{\infty\} \},
$$
where $\prod$ denotes the set of all prime numbers.

The index $S$ indicates supernatural or Steinitz. Here the exponents can take the value $\infty$, and the product
might contain an infinite number of prime numbers. If there is no danger of confusion, we will use the
simpler notation $e_p$ for the exponents. In this set only multiplication is well defined, and by definition
$$
p^{\infty}\times p^{e_p}=p^{\infty}
$$
for $ e_p \in \mathbf{Z}_0^+\cup \{\infty\}.$

In the special case that all $e_p \neq \infty$ and only a finite number of them are different from zero, then
$\prod_{p \in \prod} p^{e_p(n)} \in  N$ (i.e., $N$ is a subset of $N_S$).

{\bf Notation 6.1.} Let $(e_p)_{p\in \prod}$ (where $ e_p \in \mathbf{Z}_0^+\cup \{\infty\}$) be an infinite sequence of exponents labelled by $p \in \prod$.
Then $(e_p)_{p\in \prod}  \prec (e^{'}_p)_{p\in \prod}$  means that $e_p \le  e^{'}_p$ for all $p\in \prod$. By definition all numbers in $\mathbf{Z}_0^+$ are smaller than $\infty$.

If $(e_p)_{p\in \prod}  \prec (e^{'}_p)_{p\in \prod}$  then we say that $n =\prod_{p \in \prod} p^{e_p}$
is a divisor of $n^{'} =\prod_{p \in \prod} p^{e^{'}_p}$ and we denote it as $n|n^{'}$ or as
$n \prec n^{'}$. An element of $N_S$, corresponding to the sequence where all $e_p = \infty$, is
$$
\Omega=\prod_{p \in \prod} p^{\infty}
$$
This is the maximum element in $N_S$ (every element of $N_S$ is a divisor of $
\Omega$).

Let $\prod_1$ be a subset (finite
or infinite) of $\prod$. We set $\Omega(\prod_1)=\prod_{p \in \prod_1} p^{\infty}.$
It is obvious that $\Omega(\prod_1)|\Omega$.

{\bf Notation 6.2.}Let $n =\prod_{p \in \prod} p^{e_p(n)}$(where  $ e_p(n) \in \mathbf{Z}_0^+\cup \{\infty\}$) be
a supernatural number. We use the
notation

(ai) ~ $\prod^{(\infty)}(n)$  for the set of prime numbers for which $e_p(n) = \infty.$

(aii) ~$\prod^{(\text{fin})}(n)$ for the set of prime numbers for which $0 < e_p(n) < \infty$.

(aiii) ~$\prod^{(0)}(n)$ for the set of prime numbers for which $e_p(n)=0$.

It is obvious that $\prod=\prod^{(0)}(n) \cup  \prod^{(\text{fin})}(n) \cup \prod^{(\infty)}(n)$.

Note that
$$
n=\prod_{p \in \prod^{(\text{fin})}(n)} p^{e_p(n)}\prod_{p \in \prod^{(\infty)}(n)} p^{\infty}.
$$
If $n \in N$  then $\prod^{(\infty)}(n)$ is the empty set and $$
n=\prod_{p \in \prod^{(\text{fin})}(n)} p^{e_p(n)}.
$$
The set $N_S$ (ordered by divisibility) is a directed-complete partial order, with
 as maximum element.
An example of a complete chain in $N_S$ is

$$N_S^{(p)}=\{p, p^2, \cdot, p^{\infty}\}=N^{(p)} \cup \{p^{\infty}\}
$$
Here the supremum is $p^{\infty}$. Other examples of chains in $N_S$ are
$$
p_1 \prec p_1^2 \prec \cdots \prec p_1^{\infty} \prec  p_1^{\infty} p_2 \prec  p_1^{\infty} p^2_2
\prec  \cdots   p_1^{\infty} p_2^{\infty}; ~p_1,p_2 \in \prod.
$$
$$
p_1p_2 \prec (p_1p_2)^2 \prec \cdots \prec p_1^{\infty}p_2^{\infty};
$$
$$
2^{\infty} \prec  2^{\infty}3^{\infty} \prec  2^{\infty}3^{\infty}5^{\infty} \prec \cdots \prec \Omega.
$$

In the first and second chain the supremum is $p_1^{\infty}p_2^{\infty}$
and in the last chain the supremum is $\Omega$.

\begin{lem}Let $\ell_2$ be a Hilbert space of all square summable real valued sequences.
Let $g:N \to \ell_2$ be defined by $g(n)=ne_1$ where $e_1=(1,0,0,\cdots)$. Let $s:N_s  \to \ell_2 $ be a one-to-one mapping such that
$s|_N=g$, i.e., $s(n)=g(n)$ for each $n \in N$.  We set
$$
x \oplus_s y = s^{-1}(s(x)+ s(y))
$$
for $x,y \in N_S$. Then the operation $\oplus_s$ has the following properties:

{\bf Closure }. If $x,y \in N_S$ then $x \oplus_s y \in N_S$.

{\bf Commutativity.}~$x \oplus_s y= y \oplus_s x$ for all $x,y \in N_S$;

{\bf Associativity. }~$x \oplus_s y  \oplus_s z =(x \oplus_s y)  \oplus_s z=x \oplus_s (y  \oplus_s z)$ for all $x,y,z \in N_S$;

{\bf Extensionality  of the addition operation on $N$.}~  $n \oplus_s m= n+m$ for each $n,m \in N$.
\end{lem}
\begin{proof}{\bf Closure }. If $x,y \in N_S$ then $x \oplus_s y =s^{-1}(s(x)+ s(y)) \in N_s$.

{\bf Commutativity.} For each $x,y \in N_S $ we have
$$
x \oplus_s y = s^{-1}(s(x)+ s(y))=s^{-1}(s(y)+ s(x))=y \oplus_s x.
$$
{\bf Associativity. }~For all  $x, y,z \in N_S$, we have

$$(x \oplus_s y)\oplus_s  z=s^{-1}[s(x \oplus_s y)+
s(z)]=s^{-1}[s(s^{-1}(s(x)+ s(y)))+ s(z)]=
$$
$$
s^{-1}[s(x)+ s(y)+ s(z)]=x \oplus_s y  \oplus_s z$$
and

$$
x \oplus_s y  \oplus_s z=s^{-1}[s(x)+
(s(y)+ s(z))]= $$
$$
=s^{-1}[s(x)+
s(s^{-1}(s(y)+ s(z)))]=
$$
$$s^{-1}[s(x)+ s(y
\oplus_s z)]=x \oplus_s (y \oplus_s z).
$$

{\bf Extensionality  of the addition operation on $N$.}~ For each $n,m \in N$ we have
$$
n \oplus_s m = s^{-1}(s(x)+ s(y))=s^{-1}(ne_1+ me_1)=s^{-1}((n+m)e_1)=n+m.
$$

\end{proof}

\begin{rem} An operation $\oplus_s$ defined by Lemma 6.1 is called a plus operation in $N_S$.  Note that there does not exist  a plus operation $\oplus_s$ in $N_s$ such
$$
\underbrace{m \oplus_s \cdots \oplus_s m}_n=nm
$$
for each $n \in N$ and $m \in N_s$.
Indeed, if we assume that such a plus operation $\oplus_s$ exists in $N_s$, then we get
$$
3 \cdot 2^{\infty}= 2^{\infty}\oplus_s 2^{\infty}\oplus_s 2^{\infty}= (2^{\infty}\oplus_s 2^{\infty})\oplus_s 2^{\infty}=
2\cdot 2^{\infty} \oplus_s 2^{\infty}=2^{\infty} \oplus_s 2^{\infty}=2\cdot 2^{\infty}=2^{\infty}.
$$
The latter equality contradict to the non-equality
$$
3 \cdot 2^{\infty} \neq 2^{\infty}
$$
and Remark 6.2 is proved.
\end{rem}

{\bf Collatz conjecture for supernatural numbers.} Let $\oplus_s$ be a plus operation in $N_s$ which comes from Lemma 6.1. Take any supernatural number $n$. If $n$ is even, divide it by $2$ to get $n / 2$. If $n$ is odd, multiply it by $3$ and add $1$ to obtain $3n \oplus_s 1$.
Repeat the process infinitely. The conjecture is that no matter what supernatural number you start with, you will always eventually reach $1$.

\begin{thm} Collatz conjecture fails for supernatural numbers
\end{thm}
\begin{proof}Let take a supernatural number $n$ for which $2^{\infty} \preceq n$. Then  if we start with $n$ then we get a stationary sequence $n,n, \cdots$. Since  $2^{\infty} \preceq n$ we deduce that $n$ differs from $1$ and we no reach $1$.
\end{proof}

%\section*{Acknowledgements}
%And this is an acknowledgements section with a heading that was produced by the
%$\backslash$section* command. Thank you all for helping me writing this
%\LaTeX\ sample file.

\end{document}